\documentclass[11pt,reqno]{amsart}
\usepackage{fullpage}
\usepackage{color}
\usepackage{amssymb}
\usepackage{amsmath}
\usepackage{amsthm}
\usepackage{url}

\newtheorem{theorem}{Theorem}[section]
\newtheorem{definition}[theorem]{Definition}
\newtheorem{proposition}[theorem]{Proposition}
\newtheorem{lemma}[theorem]{Lemma}

\newtheorem{conjecture}[theorem]{Conjecture}

\newcommand{\R}{\mathbb{R}}

\newcommand{\F}{\mathbb{F}}

\begin{document}

\title{A conditional construction of restricted isometries}

\author[Bandeira]{Afonso S.~Bandeira}
\address[Bandeira]{Program in Applied and Computational Mathematics, Princeton University, Princeton, NJ, USA ({\tt ajsb@math.princeton.edu}).}

\author[Mixon]{Dustin G.~Mixon}
\address[Mixon]{Department of Mathematics and Statistics, Air Force Institute of Technology, Dayton, OH, USA ({\tt dustin.mixon@afit.edu}).}

\author[Moreira]{Joel Moreira}
\address[Moreira]{Department of Mathematics, Ohio State University, Columbus, OH, USA ({\tt  moreira@math.osu.edu}).}

\thanks{ASB was supported by AFOSR Grant No.\ FA9550-12-1-0317. DGM was supported by NSF Grant No.\ DMS-1321779. The views expressed in this article are those of the authors and do not reflect the official policy or position of the United States Air Force, Department of Defense, or the U.S.\ Government.}

\begin{abstract}
We study the restricted isometry property of a matrix that is built from the discrete Fourier transform matrix by collecting rows indexed by quadratic residues.
We find an $\epsilon>0$ such that, conditioned on a folklore conjecture in number theory, this matrix satisfies the restricted isometry property with sparsity parameter $K=\Omega(M^{1/2+\epsilon})$, where $M$ is the number of rows.
\end{abstract}

\maketitle

\section{Introduction}
Let $K\leq M\leq N$ be positive integers and let $0\leq\delta<1$.
An $M\times N$ matrix $\Phi$ is said to satisfy the \textit{$(K,\delta)$-restricted isometry property (RIP)} if
\[
(1-\delta)\|x\|^2\leq\|\Phi x\|^2\leq(1+\delta)\|x\|^2
\]
whenever $x\in\R^N$ has at most $K$ nonzero entries (i.e., $x$ is a $K$-sparse vector); here, $\|\cdot\|$ denotes the $\ell_2$ norm.
RIP matrices are important in signal processing, making it possible to measure and recover a sparse signal using significantly fewer measurements than the dimension of the signal~\cite{Candes2008}.
Random matrices have been shown to satisfy the RIP with high probability for several distributions~\cite{Baraniuk_Davenport_DeVore_Wakin08,KrahmerMR:14,NelsonPW:14,Rauhut:10,Rudelson_Vershynin08}.
However, matrices constructed randomly have a nonzero (albeit small) probability of failing to be RIP, and checking whether a given matrix satisfies this property is an NP-hard problem~\cite{Bandeira_Dobriban_Mixon_Sawin13}. This has raised the interest in constructing explicit RIP matrices~\cite{Tao07b}.

While random constructions provide sparsity levels $K$ as high as $O_\delta(M/\operatorname{polylog}N)$, which is essentially optimal \cite{Baraniuk_Davenport_DeVore_Wakin08}, most deterministic constructions only achieve $K=O(\sqrt{M})$.
The only construction so far to break this square-root bottleneck is due to Bourgain, Dilworth, Ford, Konyagin and Kutzarova~\cite{Bourgain_Dilworth_Ford_Konyagin_Kutzarova11}; they construct a matrix satisfying RIP with $K=\Omega(M^{1/2+\epsilon})$ for some $\epsilon>0$.
Their analysis has since been optimized to show that $\epsilon$ can be taken to be $4.4466\times10^{-24}$~\cite{Mixon:14}.
Some effort has also been made in derandomizing the construction of RIP matrices~\cite{Bandeira_Fickus_Mixon_Moreira15}, i.e., finding random constructions of RIP matrices using as few random bits as possible.

The \textit{Paley matrix} is a deterministic matrix constructed using the quadratic residues modulo a prime $p$; we postpone the precise definition to the next section.
In \cite{Bandeira_Fickus_Mixon_Wong13}, it was conjectured that the Paley matrix satisfies the $(K,\delta)$-RIP for some $K=\Omega_\delta(p/\operatorname{polylog} p)$.
In this note, we leverage a folklore conjecture in number theory, Conjecture~\ref{conj_chung} below, which attempts to quantify the pseudorandomness of the Legendre symbol (and hence of the Paley matrix) to prove that the Paley matrix is $(K,\delta)$-RIP with $K=\Omega(M^{1/2+\epsilon})$ for some $\epsilon>0$.
This provides another deterministic construction which breaks the square-root bottleneck, although conditionally on a conjecture.

\section{The Paley matrix}

Throughout this paper we let $p\equiv1\bmod 4$ be a prime, let $\F_p$ denote the field with $p$ elements and let $\chi:\F_p\to\{-1,0,1\}$ denote the Legendre symbol, defined by
\[
\chi(x)=\left\{\begin{array}{rl}
1&\text{if } x=y^2\text{ for some }y\in\F_p\setminus\{0\}\\
0&\text{if }x=0\\
-1&\text{otherwise.}
\end{array}\right.
\]
The \textit{Paley graph} is the graph with vertex set $\F_p$ and with an edge between two vertices $x$ and $y$ if and only if $\chi(x-y)=1$.
We now define the \textit{Paley matrix}, denoted by $\Phi$.
We use the notation $e(a):=e^{2\pi ia/p}$.

First let $Q=\{x\in\F_p:\chi(x)\geq0\}$ be the set of squares in $\F_p$, and take $M=|Q|=(p+1)/2$.
Construct the $M\times p$ matrix $H$ with entries $H[i,j]=e(-i^2j)$.
In other words, $H$ contains the rows of the discrete Fourier transform matrix indexed by $Q$.
Next let $\tilde\Phi$ be the $M\times p$ matrix obtained from $H$ by normalizing its entries so that the entries in the first row of $\tilde\Phi$ have absolute value $\sqrt{1/p}$ and the other entries of $\tilde\Phi$ have absolute value $\sqrt{2/p}$. Finally, let the Paley matrix $\Phi$ be the $M\times 2M$ matrix obtained from the concatenation of $\tilde\Phi$ with the first column of the $M\times M$ identity matrix.
For instance, the Paley matrix for $p=5$ is
\[
\Phi=\left[\begin{array}{llllll}
\sqrt{\frac15}&\sqrt{\frac15}&\sqrt{\frac15}&\sqrt{\frac15}&\sqrt{\frac15}&1\\
\sqrt{\frac25}&\sqrt{\frac25}e^{-2\pi i/5}&\sqrt{\frac25}e^{-2\pi i2/5}&\sqrt{\frac25}e^{-2\pi i3/5}&\sqrt{\frac25}e^{-2\pi i4/5}&0\\
\sqrt{\frac25}&\sqrt{\frac25}e^{-2\pi i4/5}&\sqrt{\frac25}e^{-2\pi i3/5}&\sqrt{\frac25}e^{-2\pi i2/5}&\sqrt{\frac25}e^{-2\pi i/5}&0
\end{array}\right].
\]
Interestingly, there exists a $3\times 3$ unitary matrix $U$ such that $U\Phi$ is real, and the lines spanned by the column vectors of $U\Phi$ intersect the real unit sphere at the vertices of an icosahedron; in this way, the Paley matrix generalizes the icosahedron.

One well known construction of a random matrix satisfying the RIP (with high probability) is gotten by extracting a random subset of the rows of the discrete Fourier transform matrix \cite{Rudelson_Vershynin08}.
Thus the claim that $\Phi$ satisfies the RIP can be viewed as asserting that the set $Q$ behaves randomly in this sense.
This is reasonable, as the Legendre symbol is known to behave pseudorandomly.
For instance, in \cite{Bandeira_Fickus_Mixon_Moreira15}, this pseudorandomness was used to produce RIP matrices with entries being consecutive values of $\chi$, using fewer random bits than the usual random constructions.

While currently available results concerning the random-like behavior of $\chi$ seem to be insufficient to obtain a deterministic RIP matrix, we will make use of a well known number theoretic conjecture to this end.
The following definition will be convenient:

\begin{definition}
For each of the following, we implicitly take $p\equiv1\bmod4$ to be prime:
\begin{itemize}
\item[(a)]
Let $\operatorname{PaleyDiscrepancy}[\alpha,\beta]$ denote the statement that for every sufficiently large $p$,
\begin{equation}\label{eq_conj_chung}
\bigg|\sum_{a,b\in S}\chi(a-b)\bigg|<|S|^{2-\beta}
\qquad
\forall S\subseteq\F_p\mbox{ such that }|S|>p^\alpha.
\end{equation}
\item[(b)]
Let $\operatorname{PaleyRIP}[\gamma,\eta]$ denote the statement that for every sufficiently large $p$, the $(p+1)/2\times(p+1)$ Paley matrix satisfies the $(p^\gamma,p^\eta)$-restricted isometry property.
\item[(c)]
Let $\operatorname{PaleyClique}[\tau]$ denote the statement that for every sufficiently large $p$, the largest clique in the Paley graph of $p$ vertices has $\leq p^\tau$ vertices.
\end{itemize}
\end{definition}

The following (folklore) conjecture appeared as Conjecture 2.2 in \cite{Chung94}, where a proof for $\alpha>1/2$ is given; it has also been used in \cite{Zuckerman90} to produce a pseudorandom number generator.

\begin{conjecture}\label{conj_chung}
For each $\alpha>0$, there exists $\beta=\beta(\alpha)>0$ such that $\operatorname{PaleyDiscrepancy}[\alpha,\beta]$.
\end{conjecture}

This particular formulation states that for any subcollection of vertices $S$ in the Paley graph with $|S|>p^\alpha$, the number $e(S)$ of induced edges satisfies
\[
\bigg|e(S)-\frac{1}{2}\binom{|S|}{2}\bigg|<\frac{1}{4}|S|^{2-\beta}.
\]
In particular, this implies $\operatorname{PaleyClique}[\alpha]$.
It is currently known that $\operatorname{PaleyClique}[1/2]$, although this was slightly improved in~\cite{BachocRM:13} for infinitely many $p$.
It is widely conjectured that $\operatorname{PaleyClique}[\epsilon]$ for every $\epsilon>0$.

Our main theorem puts $\operatorname{PaleyRIP}[\gamma,\eta]$ in between these two classical conjectures.
\begin{theorem}[Main Result]
\label{thm.main result}
Each of the following statements implies the next:
\begin{itemize}
\item[(a)]\label{thm.main result.a}
There exist $0<\alpha<1/2$ and $0<\beta<2$ such that $\operatorname{PaleyDiscrepancy}[\alpha,\beta]$.
\item[(b)]\label{thm.main result.b}
There exist $\gamma>1/2$ and $\tau<1/2$ such that $\operatorname{PaleyRIP}[\gamma,\tau-1/2+\epsilon]$
for every $\epsilon>0$.
\item[(c)]\label{thm.main result.c}
There exists $\tau<1/2$ such that $\operatorname{PaleyClique}[\tau+\epsilon]$ for every $\epsilon>0$.
\end{itemize}
\end{theorem}

In particular, Conjecture~\ref{conj_chung} implies that $\Phi$ satisfies the $(K,\delta)$-RIP with $K=\Omega(M^\gamma)$ for some $\gamma>1/2$.
As we will prove, it suffices to take any $\gamma$ such that
\begin{equation}\label{eq_gamma}
\frac{1}{2}
<\gamma
<\min\bigg\{\frac{1}{2-\beta},\frac{1}{4\alpha}\bigg\}.
\end{equation}
The optimal value of $\gamma$ that this approach provides is given by maximizing \eqref{eq_gamma} over $\alpha$ and with $\beta=\beta(\alpha)$ given by Conjecture~\ref{conj_chung}.
Since the optimal $\gamma$ cannot be larger than $1$, then as a byproduct of this theorem, we conclude that $\beta(\alpha)\leq1$ in Conjecture~\ref{conj_chung} whenever $\alpha<1/4$.

We postpone the proof of the implication (a)$\Rightarrow$(b) to the next section; for now, we prove only the second implication.

\begin{proof}[Proof of (b)$\Rightarrow$(c) in Theorem~\ref{thm.main result}]
Let $\omega$ be the number of vertices in the largest clique of the Paley graph.
We will show that $\omega\leq\delta\sqrt{p}$, which will prove the claim since $\delta=p^{\tau-1/2+\epsilon}$.
Let $\mathcal{K}\subset\F_p$ be a clique of maximal size, and let $a\in\F_p$ be a non-square (so that $\chi(a)=-1$).
Then $\mathcal{I}:=a\mathcal{K}$ is an independent set with cardinality $\omega$, i.e., $\chi(i-j)=-1$ for every $i,j\in \mathcal{I}$ with $i\neq j$.

Let $\varphi_0,\ldots,\varphi_p$ be the columns of $\Phi$, and let $\Phi_\mathcal{I}$ denote the submatrix of $\Phi$ containing the columns $\{\varphi_i:i\in\mathcal{I}\}$.
For any $i,j\in\F_p$ with $i\neq j$, we have
\begin{equation}\label{eq_innerproduct}
\langle \varphi_i,\varphi_j\rangle
= \frac1p+\frac2p\sum_{x=1}^{p-1}1_Q(x)\exp\big(-x(i-j)\big)
=\frac1p\sum_{y=0}^{p-1}\exp\big(-y^2(i-j)\big),
\end{equation}
where $Q$ is the set of squares in $\F_p$, and the second equality follows from the fact that each nonzero $x\in Q$ has exactly two representations as $x=y^2$.
The last sum in \eqref{eq_innerproduct} is a well-known quadratic Gauss sum, which equals $p^{1/2}\chi(i-j)$ for $p\equiv1\bmod4$.
Thus, we get
\begin{equation}\label{eq_innerproduct2}
  \langle \varphi_i,\varphi_j\rangle
  =\frac1{\sqrt{p}}\chi(i-j)
\end{equation}
In particular, if $i,j\in\mathcal{I}$ and $i\neq j$ then $\langle\varphi_i,\varphi_j\rangle=-p^{-1/2}$.
This implies that
\[
\Phi_\mathcal{I}^*\Phi_\mathcal{I}
=(1+p^{-1/2})I_\omega-p^{-1/2}J_\omega,
\]
where $I_\omega$ denotes the $\omega\times\omega$ identity matrix and $J_\omega$ denotes the matrix of all $1$s.
The smallest eigenvalue of this matrix is $1-\omega p^{-1/2}$.

By assumption, $\Phi$ satisfies the $(p^\gamma,\delta)$-RIP.
Since $|\mathcal{I}|=|\mathcal{K}|=\omega\leq p^{1/2}\leq p^\gamma$, we therefore have that the eigenvalues of $\Phi_\mathcal{I}^*\Phi_\mathcal{I}$ lie between $1-\delta$ and $1+\delta$.
In particular, $1-\omega p^{-1/2}\geq1-\delta$, and so rearranging gives $\omega\leq\delta\sqrt{p}$, as desired.
\end{proof}

\section{The Paley matrix satisfies the RIP}

In this section, we give a proof of the implication (a)$\Rightarrow$(b) in Theorem \ref{thm.main result}.
In order to show that $\Phi$ satisfies the RIP, we will employ a trick developed in~\cite{Bourgain_Dilworth_Ford_Konyagin_Kutzarova11} and show that $\Phi$ satisfies the so-called flat RIP property:

\begin{definition}
An $M\times N$ matrix $\Phi$ with columns $\varphi_1,\dots,\varphi_N$ satisfies the \textit{$(K,\theta)$-flat RIP} if for every disjoint sets $I,J\subset[N]$ such that $|J|\leq|I|\leq K$, we have
$$\bigg|\bigg\langle\sum_{i\in I}\varphi_i,\sum_{j\in J}\varphi_j\bigg\rangle\bigg|\leq\theta\sqrt{|I||J|}$$
\end{definition}

The relation between flat RIP and RIP is given by the following proposition:

\begin{proposition}\label{prop_flatRIP}
If $\Phi$ satisfies the $(K,\theta)$-flat RIP and its columns have unit norm, then it satisfies the $(K,\delta)$-RIP with $\delta=150\theta\log K$.
\end{proposition}

This proposition is essentially contained in~\cite{Bourgain_Dilworth_Ford_Konyagin_Kutzarova11}, and follows from combining Lemma~11 with Theorem~13 from \cite{Bandeira_Fickus_Mixon_Wong13}.
We will need a lemma to turn the sum of the type \eqref{eq_conj_chung} into a sum over two disjoint sets, which flat RIP uses.

\begin{lemma}\label{lemma_legendre}
Let $p\equiv1\bmod 4$ be a prime and let $\alpha>0$ and $0<\beta<2$ be such that \eqref{eq_conj_chung} holds.
Then for any $\tau\geq2\alpha/(2-\beta)$, we have
\begin{equation}\label{eq_lemma}\bigg|\sum_{i\in I,j\in J}\chi(i-j)\bigg|\leq p^\tau\sqrt{3|I||J|}\end{equation}
for any disjoint sets $I,J\subset\F_p$ with $|J|\leq|I|\leq p^{2\tau/(2-\beta)}$.
\end{lemma}

\begin{proof}
The idea is to apply \eqref{eq_conj_chung} three times: with $S=I$, $S=J$ and $S=I\cup J$.
First observe that if $|I||J|\leq3p^{2\tau}$, then we can apply the fact that $|\chi(x)|\leq1$ for all $x\in\F_p$ together with the triangle inequality to get the trivial bound:
$$\bigg|\sum_{i\in I,j\in J}\chi(i-j)\bigg|\leq|I||J|\leq\sqrt{3p^{2\tau}}\sqrt{|I||J|}=p^\tau\sqrt{3|I||J|}.$$
Thus, we will assume that
\begin{equation}\label{eq_lemma_trivialbound}|I||J|>3p^{2\tau}\end{equation}
In particular $|I|>p^\tau$.
Next, we consider the identity
$$\sum_{a,b\in I\cup J}\chi(a-b)=\sum_{a,b\in I}\chi(a-b)+\sum_{a,b\in J}\chi(a-b)+\sum_{i\in I,j\in J}\Big(\chi(i-j)+\chi(j-i)\Big).$$
Since $p\equiv1\bmod4$, for every $x\in\F_p$ we have $\chi(-x)=\chi(x)$. Thus, it follows from the triangle inequality that
$$2\bigg|\sum_{i\in I,j\in J}\chi(i-j)\bigg|\leq\bigg|\sum_{a,b\in I\cup J}\chi(a-b)\bigg|+\bigg|\sum_{a,b\in I}\chi(a-b)\bigg|+\bigg|\sum_{a,b\in J}\chi(a-b)\bigg|.$$
Applying \eqref{eq_conj_chung} with $S=I$ and again with $S=I\cup J$ (observing that $\tau>\alpha$, and so $|I\cup J|\geq|I|>p^\tau>p^\alpha$) we get
\begin{equation}\label{eq_lemma_2}2\bigg|\sum_{i\in I,j\in J}\chi(i-j)\bigg|\leq|I\cup J|^{2-\beta}+|I|^{2-\beta}+\bigg|\sum_{a,b\in J}\chi(a-b)\bigg|.\end{equation}
We now deal with the sum over $J$. If $|J|\leq p^\alpha$ we use the triangle inequality and, recalling the bound imposed on $\tau$, obtain
$$\bigg|\sum_{a,b\in J}\chi(a-b)\bigg|\leq|J|^2\leq p^{2\alpha}\leq p^{\tau(2-\beta)}\leq|I|^{2-\beta}.$$
If $|J|> p^\alpha$, then we can apply \eqref{eq_conj_chung} with $S=J$, and again we get
$$\bigg|\sum_{a,b\in J}\chi(a-b)\bigg|\leq|J|^{2-\beta}\leq|I|^{2-\beta}.$$
Plugging this back into \eqref{eq_lemma_2} and using the fact that $|I\cup J|\leq2|I|$, we get
$$\bigg|\sum_{i\in I,j\in J}\chi(i-j)\bigg|\leq\frac12|I\cup J|^{2-\beta}+|I|^{2-\beta}\leq\bigg(\frac{2^{2-\beta}}2+1\bigg)|I|^{2-\beta}<3|I|^{2-\beta}.$$
Finally, from the bound imposed on $|I|$, it follows that
$$\bigg|\sum_{i\in I,j\in J}\chi(i-j)\bigg|<3|I|^{2-\beta}\leq3p^{2\tau}=p^\tau\sqrt{9p^{2\tau}}\leq p^\tau\sqrt{3|I||J|},$$
where the last inequality follows from \eqref{eq_lemma_trivialbound}.
\end{proof}

We are now ready to prove the implication (a)$\Rightarrow$(b) of Theorem \ref{thm.main result}:

\begin{proof}[Proof of (a)$\Rightarrow$(b) in Theorem~\ref{thm.main result}]
Given $\operatorname{PaleyDiscrepancy}[\alpha,\beta]$, take $p$ large enough so that \eqref{eq_conj_chung} holds.
We will assume without loss of generality that $\beta\leq 1$, since otherwise, we may take $\beta=1$ and \eqref{eq_conj_chung} still holds.
Fix $\gamma$ satisfying \eqref{eq_gamma}.
We proceed by considering two cases:

\medskip

\noindent\textbf{Case I.} $\tfrac{2\alpha}{2-\beta}<\tfrac12$.

In this case, we pick
$$\tau=\max\bigg\{\frac{2\alpha}{2-\beta},\frac{2-\beta}2\gamma\bigg\}$$
and $K=p^{2\tau/(2-\beta)}$.
It is easily verified from \eqref{eq_gamma} that $\tau<1/2$ and $K\geq p^\gamma$.

To show that the Paley matrix $\Phi$ satisfies the RIP, we will employ Proposition \ref{prop_flatRIP} and show instead that it satisfies the flat RIP.
We index the columns of $\Phi$ by $\{0,1,\dots,p\}$.
Let $I,J$ be disjoint subsets of $\{0,1,\dots,p\}$ with $|J|\leq|I|\leq K$.
Suppose first that $p\notin I\cup J$, so that actually $I,J\subset\F_p$.
Appealing to \eqref{eq_innerproduct2}, we have
\begin{equation}\label{eq_estimate}\bigg|\bigg\langle\sum_{i\in I}\varphi_i,\sum_{j\in J}\varphi_j\bigg\rangle\bigg|=\bigg|\sum_{i\in I,j\in J}\langle\varphi_i,\varphi_j\rangle\bigg|=\frac1{\sqrt{p}}\bigg|\sum_{i\in I,j\in J}\chi(i-j)\bigg|\leq p^{\tau-1/2}\sqrt{3|I||J|},\end{equation}
where the last step is by Lemma~\ref{lemma_legendre}.
Next, consider the case where $p\in J$ (the case $p\in I$ is analogous). Recalling that the last column of $\Phi$ is $\varphi_p=[1,0,\dots,0]^\top$, we have
\[
\bigg|\bigg\langle\sum_{i\in I}\varphi_i,\sum_{j\in J}\varphi_j\bigg\rangle\bigg|
\leq\bigg|\bigg\langle\sum_{i\in I}\varphi_i,\sum_{j\in J\setminus\{p\}}\varphi_j\bigg\rangle\bigg|+\bigg|\bigg\langle\sum_{i\in I}\varphi_i,\varphi_p\bigg\rangle\bigg|
=\bigg|\bigg\langle\sum_{i\in I}\varphi_i,\sum_{j\in J\setminus\{p\}}\varphi_j\bigg\rangle\bigg|+\frac{|I|}{\sqrt{p}}.
\]
Observe that since $\beta\leq1$, $|I|\leq\sqrt{|I|}p^{\tau/(2-\beta)}\leq\sqrt{|I|}p^\tau\leq p^\tau\sqrt{3|I||J|}$. Putting this together with \eqref{eq_estimate}, we get
$$\bigg|\bigg\langle\sum_{i\in I}\varphi_i,\sum_{j\in J}\varphi_j\bigg\rangle\bigg|\leq p^{\tau-1/2}\sqrt{3|I||J|}+\frac{|I|}{\sqrt{p}}\leq2p^{\tau-1/2}\sqrt{3|I||J|}.$$

It follows that $\Phi$ satisfies the $(K,\theta)$-flat RIP with $\theta=2\sqrt{3}p^{\tau-1/2}$ and $K=p^{2\tau/(2-\beta)}\geq p^\gamma$.
By Proposition~\ref{prop_flatRIP}, we conclude that $\Phi$ satisfies the $(K,\delta)$-RIP with
$$\delta=150\theta\log K=300\sqrt{3}p^{\tau-1/2}\cdot\frac{2\tau}{2-\beta}\log p\leq p^{\tau-1/2+\epsilon},$$
where the last inequality holds for sufficiently large $p$.

\medskip

\noindent\textbf{Case II.} $\tfrac{2\alpha}{2-\beta}\geq\tfrac12$.

For this case, we will make $\beta$ smaller so that the inequality reverts to the previous case.
More precisely, pick $\gamma^*$ such that
\[
\gamma
<\gamma^*
<\min\bigg\{\frac{1}{2-\beta},\frac{1}{4\alpha}\bigg\},
\]
and pick $\beta^*=2-1/\gamma^*$.
Since $\gamma^*<1/(2-\beta)$, we get $\beta^*<\beta$ and hence \eqref{eq_conj_chung} still holds for $\beta^*$.
Since $\gamma<\gamma^*$, the inequality \eqref{eq_gamma} still holds when $\beta$ is replaced by $\beta^*$.
Since $\gamma^*<1/(4\alpha)$, we have $\tfrac{2\alpha}{2-\beta^*}<\tfrac12$.
Therefore, one may use the same proof as the previous case, but using $\beta^*$ in the place of $\beta$ (and using the same $\gamma$).
\end{proof}

\end{document}